\documentclass[11pt]{article}
\usepackage[margin=1in]{geometry} 
\usepackage{amssymb,amsfonts,amsmath,bbm,mathrsfs,stmaryrd,mathtools,mathabx}
\usepackage{xcolor}
\usepackage{url}

\usepackage{enumerate}

\usepackage{hyperref}
\hypersetup{colorlinks,
             linkcolor=black!75!red,
             citecolor=blue,
             pdftitle={},
             pdfproducer={pdfLaTeX},
             pdfpagemode=None,
             bookmarksopen=true
             bookmarksnumbered=true}

\usepackage{tikz}
\usetikzlibrary{arrows,calc,decorations.pathreplacing,decorations.markings,intersections,shapes.geometric,through,fit,shapes.symbols,positioning,decorations.pathmorphing}

\usepackage{braket}
\usepackage{tensor}

\usepackage{extarrows}

\newcommand\xlrsquigarrow[1]{%
\mathrel{%
\begin{tikzpicture}[baseline= {( $ (current bounding box.south) + (0,-0.5ex) $ )}]
  \node[inner sep=.5ex] (a) {$\scriptstyle #1$};
  \path[draw,
  >=stealth,
  <->,
  decorate,
  decoration={zigzag,amplitude=0.7pt,segment length=1.2mm,pre=lineto,pre length=4pt,post length=4pt}] 
  (a.south east) -- (a.south west);
\end{tikzpicture}}%
}

\usepackage[amsmath,thmmarks,hyperref]{ntheorem}
\usepackage{cleveref}

\creflabelformat{enumi}{#2(#1)#3}

\crefname{section}{Section}{Sections}
\crefformat{section}{#2Section~#1#3} 
\Crefformat{section}{#2Section~#1#3} 

\crefname{subsection}{\S}{\S\S}
\AtBeginDocument{%
  \crefformat{subsection}{#2\S#1#3}%
  \Crefformat{subsection}{#2\S#1#3}%
}

\crefname{subsubsection}{\S}{\S\S}
\AtBeginDocument{%
  \crefformat{subsubsection}{#2\S#1#3}%
  \Crefformat{subsubsection}{#2\S#1#3}%
}

\theoremstyle{plain}

\newtheorem{lemma}{Lemma}[section]
\newtheorem{proposition}[lemma]{Proposition}
\newtheorem{corollary}[lemma]{Corollary}
\newtheorem{theorem}[lemma]{Theorem}

\newtheorem{question}[lemma]{Question}

\theoremstyle{plain}
\theoremnumbering{Alph}
\newtheorem{theoremN}{Theorem}

\theoremstyle{plain}
\theorembodyfont{\upshape}
\theoremsymbol{\ensuremath{\blacklozenge}}

\newtheorem{definition}[lemma]{Definition}
\newtheorem{example}[lemma]{Example}

\newtheorem{remark}[lemma]{Remark}
\newtheorem{remarks}[lemma]{Remarks}

\newtheorem{notation}[lemma]{Notation}

\crefname{definition}{definition}{definitions}
\crefformat{definition}{#2definition~#1#3} 
\Crefformat{definition}{#2Definition~#1#3} 

\crefname{ex}{example}{examples}
\crefformat{example}{#2example~#1#3} 
\Crefformat{example}{#2Example~#1#3} 

\crefname{exs}{example}{examples}
\crefformat{examples}{#2example~#1#3} 
\Crefformat{examples}{#2Example~#1#3} 

\crefname{remark}{remark}{remarks}
\crefformat{remark}{#2remark~#1#3} 
\Crefformat{remark}{#2Remark~#1#3} 

\crefname{remarks}{remark}{remarks}
\crefformat{remarks}{#2remark~#1#3} 
\Crefformat{remarks}{#2Remark~#1#3} 

\crefname{convention}{convention}{conventions}
\crefformat{convention}{#2convention~#1#3} 
\Crefformat{convention}{#2Convention~#1#3} 

\crefname{notation}{notation}{notations}
\crefformat{notation}{#2notation~#1#3} 
\Crefformat{notation}{#2Notation~#1#3} 

\crefname{table}{table}{tables}
\crefformat{table}{#2table~#1#3} 
\Crefformat{table}{#2Table~#1#3}

\crefname{lemma}{lemma}{lemmas}
\crefformat{lemma}{#2lemma~#1#3} 
\Crefformat{lemma}{#2Lemma~#1#3} 

\crefname{proposition}{proposition}{propositions}
\crefformat{proposition}{#2proposition~#1#3} 
\Crefformat{proposition}{#2Proposition~#1#3} 

\crefname{corollary}{corollary}{corollaries}
\crefformat{corollary}{#2corollary~#1#3} 
\Crefformat{corollary}{#2Corollary~#1#3} 

\crefname{theorem}{theorem}{theorems}
\crefformat{theorem}{#2theorem~#1#3} 
\Crefformat{theorem}{#2Theorem~#1#3} 

\crefname{enumi}{}{}
\crefformat{enumi}{(#2#1#3)}
\Crefformat{enumi}{(#2#1#3)}

\crefname{question}{question}{Questions}
\crefformat{question}{#2question~#1#3} 
\Crefformat{question}{#2Question~#1#3} 

\crefname{assumption}{assumption}{Assumptions}
\crefformat{assumption}{#2assumption~#1#3} 
\Crefformat{assumption}{#2Assumption~#1#3} 

\crefname{construction}{construction}{Constructions}
\crefformat{construction}{#2construction~#1#3} 
\Crefformat{construction}{#2Construction~#1#3} 

\crefname{equation}{}{}
\crefformat{equation}{(#2#1#3)} 
\Crefformat{equation}{(#2#1#3)}

\numberwithin{equation}{section}
\renewcommand{\theequation}{\thesection-\arabic{equation}}

\theoremstyle{nonumberplain}
\theoremsymbol{\ensuremath{\blacksquare}}

\newtheorem{proof}{Proof}
\newcommand\pf[1]{\newtheorem{#1}{Proof of \Cref{#1}}}

\newcommand\bR{{\mathbb R}}

\newcommand\bZ{{\mathbb Z}}

\newcommand\cB{{\mathcal B}}

\newcommand\cF{{\mathcal F}}

\newcommand\cI{{\mathcal I}}
\newcommand\cJ{{\mathcal J}}
\newcommand\cK{{\mathcal K}}
\newcommand\cL{{\mathcal L}}

\newcommand\cP{{\mathcal P}}

\newcommand\numberthis{\addtocounter{equation}{1}\tag{\theequation}}

\newcommand{\qedhere}{\mbox{}\hfill\ensuremath{\blacksquare}}

\renewcommand{\square}{\mathrel{\Box}}

\title{Rich lattices of multiplier topologies}
\author{Alexandru Chirvasitu}

\begin{document}

\date{}

\newcommand{\Addresses}{{
  \bigskip
  \footnotesize

  \textsc{Department of Mathematics, University at Buffalo}
  \par\nopagebreak
  \textsc{Buffalo, NY 14260-2900, USA}  
  \par\nopagebreak
  \textit{E-mail address}: \texttt{achirvas@buffalo.edu}


}}

\maketitle

\begin{abstract}
  Each symmetrically-normed ideal $\mathcal{I}$ of compact operators on a Hilbert space $H$ induces a multiplier topology $\mu^*_{\mathcal{I}}$ on the algebra $\mathcal{B}(H)$ of bounded operators. We show that under fairly reasonable circumstances those topologies precisely reflect, strength-wise, the inclusion relations between the corresponding ideals, including the fact that the topologies are distinct when the ideals are.

  Said circumstances apply, for instance, for the two-parameter chain of Lorentz ideals $\mathcal{L}^{p,q}$ interpolating between the ideals of trace-class and compact operators. This gives a totally ordered chain of distinct topologies $\mu^*_{p,q\mid 0}$ on $\mathcal{B}(H)$, with $\mu^*_{2,2\mid 0}$ being the $\sigma$-strong$^*$ topology and $\mu^*_{\infty,\infty\mid 0}$ the strict/Mackey topology. In particular, the latter are only two of a natural continuous family. 
\end{abstract}


\vspace{.5cm}



\section*{Introduction}

The present paper was prompted by a mild puzzlement resulting from perusing some of the operator-algebra literature. Specifically, there seems to me to be a slightly misleading nomenclature conflation among several different locally convex topologies on $\cB(H)$ (bounded operators on a Hilbert space $H$), potentially confusing if, as is likely, one cross-references several sources:

\begin{itemize}
\item \cite[\S II.2]{tak1} discusses the $\sigma$-weak, $\sigma$-strong and $\sigma$-strong$^*$ and describes them in terms of seminorms involving the action of $\cB(H)$ on $H$.
\item \cite[\S I.3.1]{dixw} discusses the above-mentioned $\sigma$-weak and $\sigma$-strong topologies under different names: `ultra-weak' and `ultra-strong' respectively.
 
\item \cite[\S\S I.3.1.3 and I.3.1.6]{blk_oa} seem to suggest the terms `$\sigma$-weak', `$\sigma$-strong' and `$\sigma$-strong$^*$' have the same meaning as in \cite[\S II.2]{tak1} (and the alternative prefix `ultra' is explicitly mentioned, thus also connecting back to \cite[\S I.3.1]{dixw}), but the concrete descriptions given later on at least appear different:
  \begin{itemize}
  \item \cite[\S I.8.6.3]{blk_oa} defines the $\sigma$-strong topology as that of left multipliers on the ideal $\cK(H)\le \cB(H)$ of compact operators;
  \item while the $\sigma$-strong$^*$ topology is similarly defined as that of left and right multipliers. 
  \end{itemize}
\end{itemize}

This, of course, will later allow the identification \cite[\S II.7.3.1]{blk_oa} of the $\sigma$-strong$^*$ topology on $\cB(H)$ as defined in that book with the strict topology on the multiplier algebra
\begin{equation*}
  M(\cK(H))\cong \cB(H)
\end{equation*}
\cite[Example II.7.3.12]{blk_oa}.

This is where the apparent contradiction appears:
\begin{itemize}
\item Being a strict topology on a multiplier algebra
  \begin{equation*}
    M(\cK(H))\cong \cK(H)^{**},
  \end{equation*}
  the $\sigma$-strong$^*$ topology of \cite{blk_oa} must be {\it Mackey} \cite[Corollary 2.8]{tay_strict} (i.e. the strongest locally convex topology on $\cB(H)$ for which the predual $\cB(H)_*$ is the continuous dual).
\item On the other hand, \cite{yead_mack} shows that for infinite-dimensional $H$ the $\sigma$-strong$^*$ topology as in the {\it other} sources \cite{dixw,tak1} is {\it not} Mackey.
\end{itemize}

What seems to be happening is that despite the somewhat confusing terminological coincidence, the $\sigma$-strong$^*$ topology of \cite[paragraph following Proposition I.8.6.3]{blk_oa} is {\it not} that of \cite[Definition II.2.3]{tak1}, but rather strictly stronger (see \Cref{ex:yeadex}). Furthermore, the same goes for the two respective $\sigma$-strong topologies:

\begin{itemize}
\item The $\sigma$-strong topology of \cite[Proposition I.8.6.3]{blk_oa} is that induced by the seminorms
  \begin{equation*}
    \cB(H)\ni T\xmapsto{\quad}\|TK\|,\quad K\text{ a compact operator on }H.
  \end{equation*}
  In other words, it is the topology of left multipliers on the ideal of compact operators.
\item On the other hand, the $\sigma$-strong topology of \cite[Definition II.2.2]{tak1} (equivalently, the {\it ultra-}strong topology of \cite[\S I.3.1]{dixw}) is that induced by the seminorms
  \begin{equation*}
    \cB(H)\ni T\xmapsto{\quad}\sum_{i}\|T\xi_n\|^2\quad\text{ for $\xi_n\in H$ with }\sum_n\|\xi_n\|^2<\infty.
  \end{equation*}
  Now, assuming $\sum\|\xi_n\|^2<\infty$ ensures that
  \begin{equation*}
    e_n\xmapsto{\quad}\xi_n,\quad \text{orthonormal }(e_n)
  \end{equation*}
  extends to a {\it Hilbert-Schmidt} operator (\cite[\S I.6.6]{dixw}, \cite[Definition I.8.5.3]{blk_oa}) $K$ on $H$. Because for any orthonormal basis $(e_i)$ of $H$
  \begin{equation}\label{eq:hsnorm}
    \cL^2(H)\ni K\xmapsto{\quad} \|K\|_2:=\left(\sum_i \|Ke_i\|^2 \right)^{\frac 12}
  \end{equation}
  is a Hilbert-space norm on the ideal of Hilbert-Schmidt operators (\cite[\S I.6.6, Corollary to Theorem 5]{dixw}, \cite[\S I.8.5.5]{blk_oa}), the $\sigma$-strong topology presently being discussed (of \cite{dixw,tak1}) is that on left multipliers on the Hilbert-Schmidt ideal $\cL^2(H)$ topologized with its Hilbert-Schmidt norm \Cref{eq:hsnorm}.
\end{itemize}

Similarly:

\begin{itemize}
\item The $\sigma$-strong$^*$ topology of \cite[discussion following Proposition I.8.6.3]{blk_oa} is that of left and right multipliers on the norm-closed ideal $\cK(H)\trianglelefteq \cB(H)$ of compact operators.

\item While that of \cite[Definition II.2.3]{tak1} is the one on left and right multipliers on the Hilbert-Schmidt ideal $\cL^2(H)$.
\end{itemize}

Given all of this, together with the ubiquity of the handful of topologies typically discussed in most sources (see e.g. the reference to {\it the} eight vector-space topologies on $\cB(H)$ in \cite[\S 2.1.7]{ped_aut}), it seems worthwhile to place those topologies within some surrounding scaffolding. 

How one might go about this is already sketched above: every ideal $\cI\trianglelefteq \cB:=\cB(H)$, possibly non-norm-closed, will induce at least {\it two} topologies on $\cB$ provided it ($\cI$, that is) has been topologized in some sensible fashion: one can regard operators either as {\it left} multipliers on $\cI$, or as {\it two-sided} multipliers. The two resulting topologies will be denoted below (\Cref{def:mui}) by $\mu_{\cI}$ and $\mu^*_{\cI}$ respectively (there is a {\it right}-multiplier version, but it offers no interesting conceptual distinction to $\mu_{\cI}$).

There is a rich theory of and extensive literature on what it means for ideals $\cI\trianglelefteq \cB$ to be ``topologized in some sensible fashion'' (the operative notion being that of a {\it symmetrically normed} ideal: \cite[\S III.2]{gk_lin} and \Cref{def:sn} \Cref{item:snideal}), but the central family of examples is perhaps the chain of {\it Schatten ideals}
\begin{equation*}
  \cL^p=\cL^p(H),\ 1\le p\le \infty\quad \text{of}\quad\text{\cite[\S I.8.7.3]{blk_oa}},
\end{equation*}
consisting, respectively, of those compact operators $T$ for which the eigenvalues of the absolute value $|T|:=(T^*T)^{1/2}$ constitute an $\ell^p$-sequence.

These are naturally topologized (\Cref{def:schatp}) by norms $\|\cdot\|_p$, decreasing with $p$, and contain each other increasingly in $p$. It will also be convenient to write $\cL^{\infty}_0$ for $\cK$ itself (compact operators) and $\cL^p_0=\cL^p$ for $p<\infty$, for some notation uniformity. With all of this in place, we can state a sampling of the material below (a specialization of \Cref{th:lpqmultops}):

\begin{theoremN}
  Let $H$ be an infinite-dimensional Hilbert space, $\cL_0^p:=\cL_0^p(H)$ the corresponding $p$-Schatten ideals for $1\le p\le\infty$, and $\mu_{p\mid 0}$ and $\mu_{p\mid 0}^*$ the locally convex topologies induced on the algebra $\cB(H)$ of all bounded operators regarded, respectively, as either left multipliers or bilateral multipliers on $\cL_0^p$.
  \begin{enumerate}[(1)]
    
  \item The topologies $\mu_{p\mid 0}$ are ordered strictly increasingly in $p$ by strength; the same goes for the topologies $\mu^*_{p\mid 0}$.
    
    In fact,
    \begin{equation*}
      p \lneq q
      \xRightarrow{\quad}
      \mu_{q\mid 0}
      \npreceq
      \mu^*_{p\mid 0}.
    \end{equation*}

  \item All $\mu_{p\mid 0}$ coincide on bounded subsets of $\cB(H)$, as do all $\mu_{p\mid 0}^*$.

  \item $\mu_{p\mid 0}$ and $\mu_{p\mid 0}^*$ all have the same (norm-)bounded sets and the same (weak$^*$-)continuous functionals.  \qedhere
  \end{enumerate}
\end{theoremN}

This gives the sought-after context for the above discussion (\Cref{re:howprevtops}): $\mu_{2\mid 0}$ and $\mu^0_{2\mid 0}$ are the $\sigma$-strong and $\sigma$-strong$^*$ topologies of \cite[\S II.2]{tak1}, $\mu_{\infty\mid 0}$ is the strict/Mackey topology, and these fit into an entire chain (or rather two) of topologies, ordered {\it strictly} by $1\le p\le\infty$. Or: the familiar topologies are a few instances of a natural construction, giving a wealth of distinct topologies on $\cB(H)$.  

The Schatten ideals were here meant only as an illustration and entry point to the broader discussion. \Cref{th:lpqmultops} itself is concerned with the larger, two-parameter chain of {\it Lorentz} ideals $\cL^{p,q}$ of \cite[\S 4.11]{dfww_comm} and \cite[\S $4.2.\alpha$]{Con94} and recalled in \Cref{def:lpq}. {\it They}, in turn, are also intended to only exemplify broader phenomena.

Recall \cite[Definition 14.1]{trev_tvs} the general notion of a {\it bounded set} in a locally convex topological vector space: one absorbed by any sufficiently large scaling of any given origin neighborhood. Aggregating \Cref{pr:ifprod} and \Cref{th:topssamebdd,th:yead}, for instance, we have

\begin{theoremN}
  Let $H$ be a Hilbert space, $\cB:=\cB(H)$ and $\cK:=\cK(H)$ the ideal of compact operators.

  \begin{enumerate}[(1)]

  \item For any symmetrically-normed ideal $(\cI,\ \vvvert\cdot\vvvert)$ the left and bilateral multiplier topologies $\mu_{\cI}$ and $\mu^*_{\cI}$ all have the same bounded sets, namely the norm-bounded ones.

  \item If $\cI$ is additionally {\it approximable} (\Cref{def:schatp} \Cref{item:snapprox}) in the sense that the finite-rank operators are $\vvvert\cdot\vvvert$-dense in $\cI$, then
    \begin{itemize}
    \item $\mu_{\cI}$ coincides on bounded sets with the $\sigma$-strong topology $\mu_2$;
    \item and similarly, $\mu^*_{\cI}$ coincides on bounded sets with the $\sigma$-strong$^*$ topology $\mu^*_2$. 
    \end{itemize}

  \item If $\cI$ is approximable and also satisfies 
    \begin{equation*}
      \cL^1 = \cI (\cL^1:\cI)
    \end{equation*}
    for the ideal quotient
    \begin{equation*}
      (\cL^1:\cI):=\left\{x\in \cB\ |\ x\cI\subseteq \cL^1\right\}
    \end{equation*}
    then $\mu_{\cI}$ and $\mu_{\cI}^*$ have the same continuous dual $\cL^1\cong \cB_*$.

  \item Consider two approximable symmetrically-normed ideals $\cI, \cJ$ contained in $\cK$ such that
    \begin{equation*}
      \cI\subseteq \cJ(\cI:\cJ).
    \end{equation*}
    The left-multiplier topology $\mu_{\cI}$ is then weaker than $\mu_{\cJ}$, and similarly for the bilateral-multiplier topologies.

  \item For two approximable symmetrically-normed $\cI\not\subseteq \cJ$ contained in $\cK$ the left-multiplier topology $\mu_{\cI}$ is not dominated by the bilateral-multiplier topology $\mu^*_{\cJ}$.  \qedhere
  \end{enumerate}
\end{theoremN}

This gives, in words, some fairly generic machinery for producing distinct multiplier topologies $\mu^*_{\cI}$ (or $\mu_{\cI}$) on $\cB(H)$, whose underlying inclusion ordering {\it precisely} mimics that of the symmetrically-normed ideals $\cI$ involved in the construction.

\subsection*{Some common notation}

The symbol `$\preceq$' placed between topologies means `is coarser than'; its opposite `$\succeq$', naturally, means `is finer than'. We write $H$, $\cB(H)$, $\cK(H)$ and $\cF(H)$ for Hilbert spaces and their corresponding (possibly non-unital) rings of bounded, compact and finite-rank operators respectively.

\subsection*{Acknowledgements}

I am grateful for many insightful, informative comments and pointers to the literature by K. Davidson, A. Freslon, A. Skalski and P. Skoufranis.

This work is partially supported by NSF grant DMS-2001128.

\section{Multiplier topologies attached to ideals of $\cB(H)$}\label{se:puzz}

To distinguish between the two competing pictures of the $\sigma$-strong topology recalled above, we adorn the stronger topologies of \cite{blk_oa} with a `k' prefix, meant as suggestive of `compact': k-$\sigma$-strong and k-$\sigma$-strong$^*$. Similarly, we prepend an `hs' (for `Hilbert-Schmidt') to the corresponding topologies of \cite{dixw,tak1}: hs-$\sigma$-strong and hs-$\sigma$-strong$^*$. In either case, there is no a priori reason why
\begin{equation*}
  \text{k-$\sigma$-strong}=\text{hs-$\sigma$-strong}
  \quad\text{or}\quad
  \text{k-$\sigma$-strong$^*$}=\text{hs-$\sigma$-strong$^*$}.
\end{equation*}
As a matter of fact, the example that proves the main result of \cite{yead_mack} can be used directly to show that indeed the equalities do not hold (when $H$ is infinite-dimensional).


\begin{example}\label{ex:yeadex}
  Consider a sequence $(E_n)$ of mutually orthogonal non-zero projections on $H$. As argued in \cite{yead_mack}, the set
  \begin{equation*}
    \left\{n^{\frac 12}E_n\right\}_{n\in \bZ_{>0}},\quad E_n\in \cB(H)\text{ mutually orthogonal projections on }H
  \end{equation*}
  has $0\in  \cB(H)$ as a cluster point in the hs-$\sigma$-strong$^*$ topology of \cite[\S I.3.1]{dixw}. On the other hand, if
  \begin{equation}\label{eq:xin}
    \xi_n\in\mathrm{range}(E_n),\quad \|\xi_n\|=1,
  \end{equation}
  then
  \begin{itemize}
  \item the operator
    \begin{equation*}
      \xi_n\xmapsto{\quad}\frac 1{n^{\frac 12}}\xi_n
    \end{equation*}
    extends to a compact operator $K$ on $H$;
  \item and because
    \begin{equation*}
      E_nK \xi_n = \xi_n\Longrightarrow \|E_nK\| = 1,\ \forall n,
    \end{equation*}
    the sequence $(E_n)$ plainly does not have any subnets converging to $0$ in the k-$\sigma$-strong topology (let alone k-$\sigma$-strong$^*$). 
  \end{itemize}
\end{example}

The discussion extends considerably to provide a tower of multiplier topologies on $\cB(H)$. To place the discussion in its proper context, potentially allowing for some future extensions, we need a brief detour recalling some of the background on (possibly non-norm-closed) ideals of operators on a Hilbert space $H$. Sources for that material include \cite[Chapter III]{gk_lin}, \cite[\S 4.2 and Appendix C]{Con94}, \cite[\S\S 1.7-1.9 and Chapter 2]{sim_tr} and the briefer discussion in \cite[\S\S I.8.7.3-6]{blk_oa}.

First, following \cite[\S $4.2.\alpha$]{Con94}:

\begin{notation}\label{not:musigma}
  Let $T\in \cK(H)$ be a compact operator on a Hilbert space $H$.
  \begin{itemize}
  \item We denote by
    \begin{equation*}
      \mu_0(T)\ge\mu_1(T)\ge \cdots \ge 0
    \end{equation*}
    the {\it characteristic numbers} of $T$: the re-ordered sequence of eigenvalues of $|T|:=(T^*T)^{\frac 12}$, including multiplicities for positive eigenvalues, and padding the sequence with trailing in 0s if only finitely many eigenvalues are strictly positive.

    Observe that $H$ might be uncountably-dimensional, so the sequence $(\mu_n(T))_n$ does not quite record the multiplicity of the eigenvalue 0 accurately; this will not be an issue.

  \item $\mu(T)$ or $(\mu(T))$ or $(\mu(T))_n$ denotes the entire sequence. 

  \item Similarly, write
    \begin{equation*}
      \sigma_n(T):=\sum_{i=0}^{n-1}\mu_i(T)
    \end{equation*}
    and $\sigma(T)$ or $(\sigma(T))$, etc.
  \end{itemize}  
\end{notation}

\begin{remark}
  The $\mu_n(T)$ are the $s_j(T)$ of \cite[\S II.7 1.]{gk_lin}.   
\end{remark}

Recall, next, the {\it Schatten ideals} $\cL^p(H)\le \cB(H)$ of \cite[\S I.8.7.3]{blk_oa} (also discussed extensively in \cite[\S III.7]{gk_lin}):

\begin{definition}\label{def:schatp}
  For a Hilbert space $H$ and $1\le p< \infty$ the space $\cL^p(H)\le \cB(H)$ (or $\cL^p$, when $H$ is understood) of {\it $p$-summable} operators consists of those $T\in \cB(H)$ for which $|T|^p$ is trace-class \cite[Definition I.8.5.3]{blk_oa}:
  \begin{equation*}
    \mathrm{tr}~|T|^p
    :=
    \sum_{n\ge 0}\mu_n(|T|^p)
    =
    \sum_{n\ge 0}\mu_n(T)^p<\infty.
  \end{equation*}
  $\cL^p(H)$ is an ideal in $\cB(H)$, contained in the ideal $\cK(H)$ of compact operators, and is a Banach space under the {\it $p$-norm} $\|\cdot\|_p$ defined by 
  \begin{equation*}
    \|T\|_p^p:=\mathrm{tr}~|T|^p. 
  \end{equation*}
  Set also $\cL^{\infty}(H):=\cB(H)$ and $\cL_0^{\infty}(H):=\cK(H)$; in that case $\|\cdot\|_{\infty}$ is just the usual norm inherited from $\cB(H)$.

  It will also occasionally be convenient (and an instance of a general pattern to be revealed later) to write $\cL^p_0:=\cL^p$ for $1\le p<\infty$. 
\end{definition}

The Schatten ideals are perhaps the most prominent incarnations \cite[\S III.7]{gk_lin} of \Cref{def:sn} \Cref{item:snideal} below (for which we refer to \cite[\S\S III.2 and III.6]{gk_lin}): specialize \Cref{def:sn} \Cref{item:sn} to $\vvvert\cdot\vvvert = \|\cdot\|_p$.

\begin{definition}\label{def:sn}
  Let $H$ be a Hilbert space, $\cK:=\cK(H)\le \cB(H)=:\cB$ the ideal of compact operators thereon, and $\|\cdot\|$ the usual norm on $\cB$.
  \begin{enumerate}[(1)]

  \item\label{item:sn} A {\it symmetric norm} $\vvvert\cdot\vvvert$ on an ideal $\cI\trianglelefteq \cB:=\cB(H)$ is a norm in the usual sense \cite[Definitions 7.3 and 7.4]{trev_tvs} for which
    \begin{itemize}
    \item we also have
      \begin{equation}\label{eq:axbaxb}
        \vvvert AXB\vvvert\le \|A\|\cdot \vvvert X\vvvert \cdot \|B\|,\quad \forall X\in \cI,\ A,B\in \cB;
      \end{equation}
    \item and in addition
      \begin{equation}\label{eq:rk1}
        \vvvert X\vvvert = \|X\| = \mu_1(X)
      \end{equation}
      for all rank-1 operators $X\in \cI$. 
    \end{itemize}

  \item\label{item:snideal} A {\it symmetrically-normed (or sn-)ideal} of $\cB$ is one which is complete with respect to a symmetric norm $\vvvert \cdot\vvvert$.

  \item\label{item:sn0} For an sn-ideal $(\cI,\ \vvvert\cdot\vvvert)\trianglelefteq \cB$ we write $\cI_0\le \cI$ for the closure under $\vvvert\cdot\vvvert$ of the ideal of finite-rank operators in $\cI$.

  \item\label{item:snapprox} An sn-ideal $\cI\trianglelefteq \cB$ is {\it approximable} if $\cI_0=\cI$. 
  \end{enumerate}
  Unless specified otherwise, sn-ideals are assumed contained in $\cK$; to remind the reader of this, we will often write $\cI\trianglelefteq \cK$ (in place of $\cB$).  
\end{definition}

\begin{remark}
  Because \cite{gk_lin} works with {\it separable} Hilbert spaces, the approximable sn-ideals of that source are simply those that are separable under their symmetric norm: \cite[\S III.6, Theorems 6.1 and 6.2]{gk_lin}.

  Recall \cite[\S I.8.1.5]{blk_oa} that the ideal $\cK(H)\trianglelefteq \cB(H)$ of compact operators is the {\it largest} approximable ideal. We will typically only work with ideals
  \begin{equation*}
    \cF(K)\trianglelefteq \cI\trianglelefteq\cK(H)
  \end{equation*}
  intermediate between that of {\it finite-rank} operators and that of compact operators (these are {\it all} of the proper ideals of $\cB(H)$ when $H$ is separable \cite[\S III.1, Theorem 1.1]{gk_lin}), but note \cite[\S III.14]{gk_lin} that they are by no means all approximable. 
\end{remark}

The norms that sn-ideals come equipped with have all manner of pleasant properties \cite[\S\S III.3 and III.4]{gk_lin}:

\begin{definition}\label{def:symnorm}
  A {\it symmetric norming (or sn-)function} is a norm on the space $c_{00}$ of complex-valued eventually-zero sequences which
  \begin{enumerate}[(a)]
  \item is {\it symmetric} in the sense that $\Phi(\xi) = \Phi(\overline{\xi})$ whenever the sequence $\overline{\xi}$ is obtained from $\xi$ by permuting an initial segment thereof and then replacing the entries by their absolute values;
  \item and satisfies the normalization condition
    \begin{equation*}
      \Phi(1,\ 0,\ 0,\ \cdots)=1.
    \end{equation*}
  \end{enumerate}
\end{definition}

The correspondence between sn-ideals and sn-functions is not quite one-to-one, but there is a back-and-forth translation procedure \cite[\S III.3, Theorem 3.1]{gk_lin}:

\begin{equation}\label{eq:normfn}
  \vvvert A\vvvert_{\Phi}:=\Phi(\mu(A)):=\sup_{P}\Phi(\mu(PAP)),
\end{equation}
where $P$ ranges over finite-rank projections, defines a norm on the domain $\cI_{\Phi}$ where it is finite; that domain is furthermore an sn-ideal. The construction works backwards as well: {\it starting} with an abstract symmetric norm on the left-hand side of \Cref{eq:normfn} produces an sn-function by {\it defining} $\Phi$ via the equation. 

\begin{remark}\label{re:smlbig}
  To summarize, to an sn-function $\Phi$ one can attach canonically {\it two} sn-ideals (contained in $\cK$, per the convention mentioned in \Cref{def:sn}) whose topology is induced by the ``same'' associated norm $\vvvert\cdot\vvvert_{\Phi}$:
  \begin{itemize}
  \item the {\it largest} such ideal, denoted above by $\cI_{\Phi}$ (the $\mathfrak{S}_{\Phi}$ of \cite[\S III.4]{gk_lin});

  \item and the {\it smallest}, $\cI_{\Phi,0}=(\cI_{\Phi})_0$: the approximable portion of $\cI_{\Phi}$, denoted by $\mathfrak{S}_{\Phi}^{(0)}$ in \cite[\S III.6]{gk_lin}.
  \end{itemize}
\end{remark}

The approximable sn-ideals, it turns out, are precisely those that are ``robust'' under multiplication by ideals that are either approximable or contained in $\cK$:

\begin{theorem}\label{th:multcl}
  For an sn-ideal $\cI\trianglelefteq \cB:=\cB(H)$ on a Hilbert space $H$ the following conditions are equivalent:
  \begin{enumerate}[(a)]

  \item\label{item:ki} We have $\cK\cI=\cI(=\cI\cK)$, where $\cK$ is the ideal of compact operators.
    
  \item\label{item:ji} We have $\cJ\cI=\cI(=\cI\cJ)$ for {\it some} approximable sn-ideal $\cJ\trianglelefteq \cB$.
    
  \item\label{item:ink} We have $\cJ\cI=\cI(=\cI\cJ)$ for some ideal of $\cB$ contained in $\cK$.
    
  \item\label{item:isprod} $\cI$ is a product $\cJ\cI'$ for sn-ideals $\cI'$ and $\cJ$ at least one of which is approximable. 
 
  \item\label{item:isapprox} $\cI$ is approximable.

  \end{enumerate}
\end{theorem}

A brief postponement of the proof will render it simpler.

The multiplier terminology below is of course familiar; the notation is borrowed from \cite[discussion following Proposition 1.11]{am_comm}. 

\begin{definition}\label{def:multide}
  Let $\cI,\cJ\trianglelefteq \cB:=\cB(H)$ be two ideals. The corresponding {\it ideal quotient}, or the ideal of {\it $\cJ\to \cI$ multipliers} is
  \begin{equation*}
    (\cI:\cJ):=\{x\in \cB\ |\ x\cJ\subseteq \cI\}
    =\{x\in \cB\ |\ \cJ x\subseteq \cI\}.
  \end{equation*}
  The left-right symmetry follows from the fact that ideals in $\cB$ are automatically self-adjoint \cite[\S III.1, following Corollary 1.1]{gk_lin}.

  Given a third ideal $\cL\trianglelefteq \cB$ we will also write
  \begin{equation*}
    (\cI:\cJ)_{\cL}:=(\cI:\cJ)\cap \cL;
  \end{equation*}
  this applies mostly to $\cL:=\cK$, the ideal of compact operators, given that most ideals of interest in the sequel will be contained in $\cK$ to begin with. 
\end{definition}

\begin{remark}
  The ideal quotients $(\cI:\cJ)$ are the {\it relative conjugate ideals} $\cJ^{\times}\braket{\cI}$ of \cite[Definition 3.3]{sal_symm}. 
\end{remark}

The following simple remark will come in handy frequently. 

\begin{proposition}\label{pr:multarecont}
  Let $(\cI,\ \vvvert\cdot\vvvert_{\cI})$ and $(\cJ,\ \vvvert\cdot\vvvert_{\cJ})$ be two sn-ideals in $\cK:=\cK(H)$ for some Hilbert space $H$.

  For every $x\in (\cI:\cJ)$ the map
  \begin{equation*}
    (\cI,\ \vvvert\cdot\vvvert_{\cI})
    \xrightarrow{\quad x\cdot\quad}
    (\cJ,\ \vvvert\cdot\vvvert_{\cJ})
  \end{equation*}
  is continuous, as is the corresponding {\it right} multiplication.   
\end{proposition}
\begin{proof}
  The multiplication map is of course continuous with respect to the original norm $\|\cdot\|$, so the graph of the map is a closed subset of $\cI\times \cJ$ for {\it that}, weaker topology. It is thus also closed for the stronger topology induced by the two norms $\vvvert\cdot\vvvert_{\cI,\cJ}$, and we are done by the {\it Closed Graph theorem} \cite[Corollary 4 to Theorem 17.1]{trev_tvs}.
\end{proof}

\pf{th:multcl}
\begin{th:multcl}
  The achirality of conditions \Cref{item:ki}-\Cref{item:isprod} follows from the fact that all ideals in sight are self-adjoint, so we will not bother with the issue of left versus right multiplication much.
  \begin{enumerate}[]
  \item {\bf \Cref{item:ki} $\Longleftrightarrow$ \Cref{item:ji} $\Longleftrightarrow$ \Cref{item:ink}} follows from the fact that $\cK$ is the {\it largest} approximable ideal, so that 
    \begin{equation*}
      \cJ\cI\subseteq \cK\cI\subseteq \cI
    \end{equation*}
    for any approximable $\cJ$. 

  \item {\bf \Cref{item:ink} $\Longrightarrow$ \Cref{item:isprod}:} immediate.

  \item {\bf \Cref{item:isprod} $\Longrightarrow$ \Cref{item:isapprox}.} Every element of $\cI$ is by assumption a finite sum of products $x_{\cJ}x_{\cI'}$ of elements belonging to the two ideals depicted as subscripts, so it is enough to argue that such products are approximable.
    
    Now, $x_{\cI'}$ belongs to the ideal quotient $(\cI:\cJ)$, so by \Cref{pr:multarecont} the resulting multiplication map
    \begin{equation*}
      (\cJ,\ \vvvert\cdot\vvvert_{\cJ})
      \xrightarrow{\quad \cdot x_{\cI'}\quad}
      (\cI,\ \vvvert\cdot\vvvert_{\cI})
    \end{equation*}
    is continuous. Because $x_{\cJ}$ is (by hypothesis) $\vvvert \cdot\vvvert_{\cJ}$-approximable by finite-rank operators, the same goes for $x_{\cJ}x_{\cI'}$ in the $\vvvert \cdot\vvvert_{\cI}$ norm.

  \item {\bf \Cref{item:isapprox} $\Longrightarrow$ \Cref{item:ki}.} The {\it polar decomposition}
    \begin{equation*}
      \cI\ni x=u|x|,\quad |x|=u^*x
    \end{equation*}
    of \cite[\S I.5.2.2]{blk_oa} affords a decomposition of an arbitrary element of $\cI$ as a product with a {\it positive} factor from $\cI$, so it is enough to work with positive elements $0\le x\in \cI$. Being also compact, these can be diagonalized:
    \begin{equation*}
      x=\mathrm{diag}(\mu_0(x),\ \mu_1(x),\ \cdots).
    \end{equation*}
    Approximability implies via \cite[\S III.6, Lemma 6.1]{gk_lin} that
    \begin{equation}\label{eq:mononorm}
      \Phi(\mu_{n+1}(x),\ \mu_{n+2}(x),\ \cdots)\xrightarrow[n]{\quad}0,
    \end{equation}
    where $\Phi$ is the sn-function implementing the norm $\vvvert\cdot\vvvert_{\cI}$. Select indices $n_k$ such that
    \begin{equation*}
      \Phi(\mu_{n_k}(x),\ \mu_{n_k+1}(x),\ \cdots) < \frac 1{4^k}. 
    \end{equation*}
    Scaling each $\mu_{\bullet}(x)$ for
    \begin{equation*}
      n_k\le \bullet\le n_{k+1}-1
    \end{equation*}
    by $2^k$ will produce a sequence $(\mu'_n)_n$ that still satisfies \Cref{eq:mononorm}, so that
    \begin{equation*}
      x'=\mathrm{diag}(\mu'_0(x),\ \mu'_1(x),\ \cdots)\in \cI
    \end{equation*}
    by the completeness of $\cI$ under the norm induced by $\Phi$. Observe next that $x=tx'$ for a diagonal operator $t$ whose entries with indices between $n_k$ and $n_{k+1}-1$ inclusive are $\frac 1{2^k}$. In particular $t$ is compact, and we are done: $x\in \cK \cI$.  \qedhere
  \end{enumerate}
\end{th:multcl}

\begin{remark}\label{re:salinas}
  The discussion following \cite[\S III.4, Theorem 4.1]{gk_lin} shows that two ideals $\cI_{\Phi_i}$ (\Cref{re:smlbig}) attached to sn-functions $\Phi_i$, $i=1,2$ coincide precisely when the functions are equivalent in the usual norm-theoretic sense (see also \cite[\S III.2, Theorem 2.1]{gk_lin}). \cite[Theorem 2.3]{sal_symm} extends this remark to {\it inclusions} of sn-ideals (that paper also works with separable Hilbert spaces, but the proof simply goes through in general):
  
For two sn-functions $\Phi_i$, $i=1,2$ the following conditions are equivalent.
  \begin{enumerate}[(a)]
  \item\label{item:whenincl-a} $\Phi_2\le C\Phi_1$ for some constant $C>0$.
    
  \item\label{item:whenincl-a'} $\vvvert\cdot\vvvert_{\Phi_2}\le C\vvvert\cdot\vvvert_{\Phi_1}$ for some constant $C>0$.

  \item\label{item:whenincl-b} $\cI_{\Phi_1}\subseteq \cI_{\Phi_2}$.
  \item\label{item:whenincl-c} $\cI_{\Phi_1,0}\subseteq \cI_{\Phi_2}$.
  \item\label{item:whenincl-d} $\cI_{\Phi_1,0}\subseteq \cI_{\Phi_2,0}$.
  \item\label{item:whenincl-e} Some $\vvvert\cdot\vvvert_{\Phi_1}$-normed sn-ideal $\cI_1$ is contained in some $\vvvert\cdot\vvvert_{\Phi_2}$-normed sn-ideal $\cI_2$.
  \end{enumerate}  
\end{remark}

To circle back to multipliers and their attendant topologies:

\begin{definition}\label{def:mui}
  Let $H$ be a Hilbert space and $(\cI,\ \vvvert\cdot\vvvert_{\cI})\trianglelefteq \cB:=\cB(H)$ be an sn-ideal.
  \begin{enumerate}[(1)]

  \item\label{item:muli} The {\it (left) $\cI$-multiplier topology} (or {\it multiplier topology attached to $\cI$}) $\mu_{\cI}=\tensor[_\ell]{\mu}{_{\cI}}$ on $\cB$ is the locally convex topology induced by the seminorms
    \begin{equation}\label{eq:muli}
      \cB\in x\xmapsto{\quad} \|xx_{\cI}\|_{\cI},\quad x_{\cI}\in \cI. 
    \end{equation}

  \item\label{item:muri} Similarly, the {\it right} $\cI$-multiplier topology $\tensor[_r]{\mu}{_\cI}$ is induced by the seminorms 
    \begin{equation}\label{eq:muri}
      \cB\in x\xmapsto{\quad} \|x_{\cI} x\|_{\cI},\quad x_{\cI}\in \cI
    \end{equation}
    instead.

    As the notation suggests, unless the `left' or `right' modifiers are invoked explicitly, we default to `left'.

  \item\label{item:muast} The {\it $\cI$-multiplier$^*$ topology} $\mu^*_{\cI}$ is the coarsest locally convex topology finer than both $\mu_{\ell,\cI}$ and $\mu_{r,\cI}$. In other words, it is induced by {\it both} \Cref{eq:muli} and \Cref{eq:muri}.     
  \end{enumerate}
\end{definition}

\begin{remark}
  The `*' in \Cref{def:mui} \Cref{item:muast} is there (by analogy, say, with the strong$^*$ or $\sigma$-strong$^*$ topology \cite[\S I.3.1.6]{blk_oa}) in order to recall that convergence of a net $(x_{\lambda})_{\lambda}$ in $\mu_{\cI}^*$ is equivalent to $\mu_{\ell,\cI}$-convergence of both $(x_{\lambda})$ and $(x^*_{\lambda})$.
\end{remark}

The ideal quotients of \Cref{def:multide} play a role in assessing multiplier-topology strength.

\begin{proposition}\label{pr:ifprod}
  Let $\cI,\cJ\trianglelefteq \cB:=\cB(H)$ be two sn-ideals.

  If $\cI=\cJ (\cI:\cJ)$ (or, equivalently, $\cI\subseteq \cJ (\cI:\cJ)$) then
  \begin{equation}\label{eq:mule}
    \mu_{\cI}\preceq \mu_{\cJ}
    ,\quad
    \tensor[_r]{\mu}{_\cI}\preceq \tensor[_r]{\mu}{_\cJ}
    \quad\text{and}\quad
    \mu^*_{\cI}\preceq \mu^*_{\cJ}.
  \end{equation}
\end{proposition}
\begin{proof}
  That equality is equivalent to inclusion follows from the fact that the {\it opposite} inclusion
  \begin{equation*}
    \cI\supseteq \cJ (\cI:\cJ)
  \end{equation*}
  is automatic.

  Proving the first dominance in \Cref{eq:mule} will suffice, as the others follow (the second is obtained by applying the $*$ operation to the first, etc.).

  Suppose
  \begin{equation*}
    \cI\ni x_{\cI} = x_{\cJ}y,\quad x_{\cJ}\in \cJ,\ y\in (\cI:\cJ).
  \end{equation*}
  If
  \begin{equation*}
    x_{\lambda}x_{\cJ}\xrightarrow[\lambda]{\quad\vvvert\cdot\vvvert_{\cJ}\quad} xx_{\cJ}
  \end{equation*}
  for a net $(x_{\lambda})\subset \cB$, then
  \begin{equation*}
    x_{\lambda}x_{\cI}=x_{\lambda}x_{\cJ}y\xrightarrow[\lambda]{\quad\vvvert\cdot\vvvert_{\cI}\quad} xx_{\cJ}y = xx_{\cI}
  \end{equation*}
  by \Cref{pr:multarecont}. The conclusion follows from the assumption that all elements of $\cI$ are sums of products $x_{\cJ}y$.
\end{proof}

In particular:

\begin{corollary}\label{cor:allweakerthanmuinf}
  For all approximable sn-ideals $\cI\trianglelefteq \cK(H)$ we have
  \begin{equation*}
    \mu_{\cI}\preceq \mu_{\cK}
    \quad\text{and}\quad
    \mu^*_{\cI}\preceq \mu^*_{\cK}.
  \end{equation*}
\end{corollary}
\begin{proof}
  This is a consequence of \Cref{pr:ifprod}, given that
  \begin{equation*}
    \cI = \cI\cK = \cI(\cI:\cK)_{\cK}\subseteq \cI(\cI:\cK)
  \end{equation*}
  by \Cref{th:multcl}.
\end{proof}

\begin{remark}\label{re:orderincl}
  The relation
  \begin{equation}\label{eq:ijij}
    \cI\subseteq \cJ(\cI:\cJ)
  \end{equation}
  is transitive, as follows easily from
  \begin{equation*}
    (\cI:\cJ)(\cJ:\cJ') \subseteq (\cI:\cJ').
  \end{equation*}
  The same goes for the relative ideal quotients $(-:-)_{\cL}$ of \Cref{def:multide}. Said relations are also antisymmetric, because they imply $\cI\subseteq \cJ$. Reflexivity is furthermore obvious for \Cref{eq:ijij} (because $1\in (\cI:\cI)$) and holds for $(-:-)_{\cK}$ for approximable ideals by \Cref{th:multcl}, so in the cases of interest, at least, the relations in question are partial orders on the collection of ideals in $\cB(H)$. For that reason, we henceforth write
  \begin{equation}\label{eq:orderincl}
    \cI\preceq_{\cL} \cJ
    \quad\text{for}\quad
    \cI\subseteq \cJ(\cI:\cJ)_{\cL}
  \end{equation}
  for brevity. 
\end{remark}

\Cref{def:schatp} extends to a two-parameter family of ideals/norms (\cite[\S $4.2.\alpha$]{Con94}, \cite[\S I.8.7.6]{blk_oa}):

\begin{definition}\label{def:lpq}
  Let $H$ be a Hilbert space,
  \begin{equation*}
    1<p<\infty,\quad 1\le q\le \infty
    \quad\text{or}\quad
    p,q\in \{1,\infty\},
  \end{equation*}
  and write $\alpha:=\frac 1p$, $\beta:=\frac 1q$.
  
  The {\it Lorentz ideal} (\cite[\S 4.11]{dfww_comm}, \cite[\S 1.3]{bl_interp}) $\cL^{p,q}=\cL^{p,q}(H)$, complete with respect to the corresponding norm $\|\cdot\|_{p,q}$, is defined as follows (in overlapping cases conflicting norm definitions are equivalent):
  \begin{enumerate}[(a)]
  \item\label{item:deflp} $\cL^{p,p}=\cL^p$ equipped with the usual norm $\|\cdot\|_p$ of \Cref{def:schatp} for all $1\le p\le\infty$;
    
  \item\label{item:deflpq} as the natural domain of the norm
    \begin{equation}\label{eq:pqnorm}
      \|T\|_{p,q}:=\left(\sum_{n\ge 1}n^{(\alpha-1)q-1}\sigma_n(T)^q\right)^{1/q}
    \end{equation}
    for
    \begin{equation*}
      1<p<\infty,\ 1\le q<\infty
      \quad\text{or}\quad
      p=\infty,\ q=1;
    \end{equation*}

  \item\label{item:deflpinf} as the natural domain of the norm \cite[p.9]{sim_tr}
    \begin{equation}\label{eq:pinfnorm}
      \|T\|_{p,\infty}:=\sup_{n\ge 1}n^{\alpha-1}\sigma_n(T)
    \end{equation}
    for $1<p$, $q=\infty$;

  \item\label{item:defl1inf} finally, as the natural domain of the norm
    \begin{equation}\label{eq:1infnorm}
      \|T\|_{1,\infty}:=\sup_{n\ge 2}\frac{\sigma_n(T)}{\log n}
    \end{equation}
    in the remaining case $p=1$, $q=\infty$. 
  \end{enumerate}
\end{definition}

\begin{remarks}\label{res:whennot0}
  \begin{enumerate}[(1)]

  \item That \Cref{eq:pqnorm,eq:pinfnorm,eq:1infnorm} are indeed norms follows from the fact that the $\sigma_n$ are \cite[(4.9)]{Con94}, as are the $L^q$ norms for $1\le q\le\infty$ for arbitrary measure spaces \cite[Theorem 3.9]{rud_rc} (e.g. $\left(\bZ_{>0},\ \nu\right)$ equipped with the measure $\nu(n)=n^{(\alpha-1)q-1}$, as appropriate for \Cref{eq:pqnorm}). 
    
  \item As indicated in passing in \Cref{def:lpq}, the terminology is borrowed from the {\it Lorentz spaces} \cite[\S 5.3]{bl_interp} familiar from interpolation theory. 

  \item There are related notions in the literature, such as the {\it Lorentz sequence spaces} $d({\bf w},p)$ of \cite[Vol.1, \S 4.e]{lt_cls}, where
    \begin{itemize}
    \item ${\bf w}=(w_n)_{n\in \bZ_{>0}}$ is a non-increasing positive sequence with
      \begin{equation*}
        w_1=1,\quad \lim_n w_n=0,\quad \sum_n w_n=\infty
      \end{equation*}
    \item and $1\le p<\infty$. 
    \end{itemize}
    The corresponding sequence space $d({\bf w},p)$ is a Banach space under the norm
    \begin{equation*}
      \|(\xi_n)\|_{{\bf w},p}:=\sup_{\text{permutations }\pi}\left(\sum_{n}|\xi_{\pi(n)}|^p w_n\right)^{\frac 1p}:
    \end{equation*}
    a weighted $\ell^p$ space of sorts. 

    One can construct a corresponding ideal in any $\cB(H)$ (and in fact $\cK(H)$), by the usual \cite[\S I.8.7.4]{blk_oa} correspondence between {\it order ideals} \cite[Definition 2.20]{dp_latt} in the lattice
    \begin{equation*}
      c^+_{0,\searrow}:=\left\{\text{non-decreasing, non-negative sequences vanishing at }0\right\}:
    \end{equation*}
    the collection
    \begin{equation*}
      d({\bf w},p)^+_{\searrow}
      :=
      \left\{\text{non-decreasing, non-negative members of }d({\bf w},p)\right\}
    \end{equation*}
    is precisely such an order ideal, and the corresponding (sn-)ideal in $\cK(H)$ is
    \begin{equation*}
      \left\{T\in \cK(H)\ |\ (\mu_n(T))_n\in d({\bf w},p)^+_{\searrow}\right\}.
    \end{equation*}
    \cite[\S III.15]{gk_lin} also discusses these for $p=1$: there, they would be denoted by $\mathfrak{S}_{\bf w}$. 

    Some of the $\cL^{p,q}$ can be obtained in this fashion, as will become apparent after \Cref{pr:hardy}, but not all: per \cite[Vol.1, discussion following Definition 4.e.1]{lt_cls}, $\ell^p$ itself is not even isomorphic to any of the $d({\bf w},p)$.
    
  \item\label{item:whenneed0} The ideal $\cL^{p,q}$ fails to be approximable precisely when $q=\infty$ (see either \cite[\S\S $4.2.\alpha$ and $4.2.\beta$, pp.309-310]{Con94} or \cite[\S I.8.7.6]{blk_oa}, though neither source states this fully).

    The corresponding approximable sub-ideals are as follows:
    \begin{itemize}
    \item for $1\le p<\infty$,
      \begin{equation*}
        \cL^{p,\infty}_0
        =
        \left\{T\in \cL^{p,\infty}\ \bigg|\ n^{\alpha-1}\sigma_n(T)\xrightarrow[n]{\quad}0\right\}
        =
        \left\{T\in \cL^{p,\infty}\ \bigg|\ n^{\alpha}\mu_n(T)\xrightarrow[n]{\quad}0\right\};
      \end{equation*}
    \item and
      \begin{equation*}
        \cL^{1,\infty}_0
        =
        \left\{T\in \cL^{1,\infty}\ \bigg|\ \frac{\sigma_n(T)}{\log n}\xrightarrow[n]{\quad}0\right\}
      \end{equation*}
    \end{itemize}
  \end{enumerate}
\end{remarks}

\begin{notation}\label{not:mutop}
  For the ideals $\cI=\cL^p$ or $\cL^{p,q}$ we denote the multiplier topologies of \Cref{def:mui} \Cref{item:muli} by $\mu_p$ and $\mu_{p,q}$ respectively, and similarly for parts \Cref{item:muri} and \Cref{item:muast} ($\tensor[_r]{\mu}{_p}$, $\mu_{p,q}^*$, etc.)

  To distinguish between the multiplier topologies induced by $\cL^{p,q}$ and $\cL^{p,q}_0$ we decorate the respective symbol with an additional `0' subscript, as in $\mu_{p,q\mid 0}$ or $\mu_{p\mid 0}^*$.
\end{notation}

\begin{remark}\label{re:howprevtops}
  The connection to the preceding discussion is the observation that
  \begin{itemize}
  \item $\mu_2$ and $\mu_2^*$ are precisely the $\sigma$-strong and $\sigma$-strong$^*$ topologies respectively;

  \item while $\mu_{\infty\mid 0}^*$ is nothing but the strict (hence also Mackey \cite[Corollary 2.8]{tay_strict}) topology on $\cB(H)\cong M(\cK(H))$.
  \end{itemize}
\end{remark}

It being occasionally more convenient to work with the sequence $(\mu_n(T))$ rather than $(\sigma_n(T))$, we record the following remark that seems a little difficult to locate in the literature in directly-citable form.

The point is that in the generic summand of \Cref{eq:pqnorm}, one can harmlessly make a substitution
\begin{equation*}
  \left(\frac{\sigma_n(T)}n\right)^q
  \xrightarrow{\qquad}
  \mu_n(T)^q
\end{equation*}
and similarly, every $\sigma_n$ can be replaced by an $n\mu_n$ in \Cref{eq:pinfnorm}. This is a very familiar theme, and an application of one of the myriad versions of {\it Hardy's inequality} \cite[\S 1.1]{ok_hardy}. The general principle is also invoked by name on \cite[p.308]{Con94} in a more limited scope, to conclude that $\|\cdot\|_{p,p}$ and $\|\cdot\|_p$ give rise to the same ideal.

Recall \cite[\S 1.3, p.7]{bl_interp} that a {\it quasi-}norm differs from a norm in that it might satisfy the triangle inequality only up to scaling:
\begin{equation*}
  \|x+y\|\le C(\|x\|+\|y\|).
\end{equation*}
\cite[Example preceding Theorem 1.16, p.9]{sim_tr} explains why the `quasi' qualifier is needed, say, for \Cref{eq:pinfnorm'} below.

\begin{proposition}\label{pr:hardy}
  In cases \Cref{item:deflpq} and \Cref{item:deflpinf} of \Cref{def:lpq} the norms \Cref{eq:pqnorm} and \Cref{eq:pinfnorm} are equivalent to the quasi-norms
  \begin{equation}\label{eq:pqnorm'}
    \|T\|'_{p,q}:=\left(\sum_{n\ge 1}n^{\alpha q-1}\mu_{n-1}(T)^q\right)^{1/q}
  \end{equation}
  and
  \begin{equation}\label{eq:pinfnorm'}
    \|T\|'_{p,\infty}:=\sup_{n\ge 1}n^{\alpha}\mu_{n-1}(T)
  \end{equation}
  respectively, where $\alpha:=\frac 1p$.

  Consequently, these define the same ideals $\cL^{p,q}$ of \Cref{def:lpq}.
\end{proposition}
\begin{proof}
  We drop $T$ to streamline the notation, since this is essentially a claim about sequences $(\mu_n)$.
  
  Since by definition
  \begin{equation*}
    \frac{\sigma_n}{n} = \frac{\sum_0^{n-1}\mu_i}n\ge \mu_n
  \end{equation*}
  because $(\mu_n)_n$ is non-increasing, we certainly always have $\|\cdot\|\ge \|\cdot\|'$. The substance of the claim is thus the fact that there is an inequality of the form
  \begin{equation*}
    \|\cdot\|_{p,q}\le C\|\cdot\|'_{p,q}\quad\text{for some constant } C=C_{p,q}>0 .
  \end{equation*}
  For \Cref{eq:pinfnorm'} this follows from
  \begin{equation*}
    \int_1^N x^{-\alpha}\ \mathrm{d}x = \frac{N^{1-\alpha}-1}{1-\alpha}\le C N^{1-\alpha}.
  \end{equation*}
  We can henceforth focus on the case
  \begin{equation*}
    1<p<\infty,\ 1\le q<\infty
    \quad\text{or}\quad
    p=\infty,\ q=1.
  \end{equation*}
  To dispose of the ``boundary'' cases $q=1$ first, note that
  \begin{equation*}
    \begin{aligned}
      \sum_{n\ge 1}\frac {\sigma_n}{n^{2-\alpha}}
      =
      \sum_{n\ge 1}\frac {\mu_0+\cdots+\mu_{n-1}}{n^{2-\alpha}}
      &=
        \sum_{n\ge 1}\left(\frac 1{n^{2-\alpha}}+\frac 1{(n+1)^{2-\alpha}}+\cdots\right)\mu_{n-1}\\
      &\le
        C\sum_{n\ge 1}\frac{\mu_{n-1}}{n^{1-\alpha}}
    \end{aligned}    
  \end{equation*}
  for appropriate $C>0$ because
  \begin{equation*}
    \int_N^{\infty}x^{\alpha-2}\ \mathrm{d}x=\frac{N^{\alpha-1}}{1-\alpha}. 
  \end{equation*}
  The remaining bulk of the claim now follows from \cite[Theorem 2]{lfv_hardy}, with some processing:
  \begin{itemize}
  \item that source's $p$ is our $q$ (so that it is indeed {\it strictly} larger than $1$, as \cite[Theorem 2]{lfv_hardy} requires, in the regime currently being considered);
  \item the $\alpha$ of \cite[Theorem 2]{lfv_hardy}, here, is
    \begin{equation*}
      \alpha-\frac 1q = \frac 1p - \frac 1q
    \end{equation*}
    instead;
  \item and the $\frac{a_k}{(k+1)^{\alpha}}$ of loc.cit. would be our $\mu_{k+1}$. 
  \end{itemize}
  The theorem in question functions in the range
  \begin{equation*}
    \frac 1p - \frac 1q\in \left[0,\ 1-\frac 1q\right).
  \end{equation*}
  The upper bound is no issue (we have it in any case, because $\frac 1p<1$), but the lower boundary condition
  \begin{equation*}
    \frac 1p - \frac 1q\ge 0
    \iff
    p\le q
  \end{equation*}  
  is \cite{lfv_comm}, apparently, crucial to the proof of \cite[Theorem 2]{lfv_hardy}. Note however that if \cite[equation (CWHI)]{lfv_hardy} holds for some $\alpha$, then it still does, with the same right-hand scaling constant, for {\it smaller} $\alpha$ on the left: the scaling factors
  \begin{equation*}
    \frac{(n+1)^{\alpha}}{(k+1)^{\alpha}},\ 0\le k\le n
  \end{equation*}
  appearing in the left-hand side are non-decreasing functions of $\alpha$.
\end{proof}

\begin{remarks}
  \begin{enumerate}[(1)]
  \item That \Cref{eq:pqnorm'} and \Cref{eq:pinfnorm'} are indeed quasi-norms {\it follows} from \Cref{pr:hardy}: homogeneity and non-degeneracy is not an issue, whereas the subadditivity-up-to-scaling characteristic of quasi-norms is a consequence of the honest subadditivity
    \begin{equation*}
      \|\square + \bullet\|\le \|\square\| + \|\bullet\|
    \end{equation*}
    together with the estimates
    \begin{equation*}
      C_0 \|\cdot\| \le \|\cdot\|'\le C_1 \|\cdot\|
    \end{equation*}
    provided by \Cref{pr:hardy}.

  \item It is tempting to push \Cref{eq:pinfnorm'} to its limiting case $p=\alpha=1$. Considering its $\sigma$ counterpart \Cref{eq:1infnorm} and the fact (e.g. \cite[\S 1.5]{finch_const-1}) that
    \begin{equation}\label{eq:euler}
      1+\frac 12+\cdots+\frac 1n\sim \log n,
    \end{equation}
    this suggests
    \begin{equation}\label{eq:norm1inf'cand}
      \|T\|'_{1,\infty}:=\sup_{n\ge 2} n\mu_n(T).
    \end{equation}
    While \Cref{eq:euler} does imply that we have an inequality of the form $\|\cdot\|\le C\|\cdot\|'$ (the interesting part of the proof of \Cref{pr:hardy}), in this case we cannot take for granted the {\it opposite} inequality, which in fact does not hold: choosing
    \begin{equation*}
      \mu_n=\mu_n(T)=
      \begin{cases}
        \frac 1{n_0}&\text{ for }n\le n_0\\
        0&\text{ otherwise}
      \end{cases}
    \end{equation*}
    for large $n_0$ will achieve arbitrarily small
    \begin{equation*}
      \|T\|_{1,\infty}
      =
      \sup\frac{\sigma_n(T)}{\log n}
      \sim
      \sup\frac{\mu_0+\cdots+\mu_{n-1}}{1+\cdots+\frac 1{n-1}}
    \end{equation*}
    while \Cref{eq:norm1inf'cand} is 1.

  \item The un-definability of $\cL^{1,\infty}$ by \Cref{eq:norm1inf'cand} is intimately connected to the fact that the sequence $(n^{-1})_n$ is not {\it regular} in the sense of \cite[\S III.14, 3.]{gk_lin}: its partial sums \Cref{eq:euler} are not bounded above by (some constant times) $n\cdot \frac 1n=1$.

    By contrast, and in reference to \Cref{eq:pinfnorm'}, note the remark preceding \cite[\S III.14, Theorem 14.2]{gk_lin} to the effect that $(n^{-\alpha})_n$ {\it is} regular for $0<\alpha<1$. 
  \end{enumerate}  
\end{remarks}

It will also be useful to record the inclusion relations between the various $\cL^{p,q}$. \cite[\S $4.2.\alpha$, Proposition 1]{Con94} does this (presumably for $1<p<\infty$) via a general appeal to interpolation theory, but the claim is easy to prove by an elementary direct examination of the various cases involved.

\begin{lemma}\label{le:howcontain}
  The inclusions between the various $\cL^{p,q}$ of \Cref{def:lpq} are lexicographic in $(p,q)$:
  \begin{equation*}
    \cL^{p_0,q_0}\subseteq \cL^{p_1,q_1}
  \end{equation*}
  if and only if either $p_0<p_1$ or $p_0=p_1$ and $q_0\le q_1$.  \qedhere
\end{lemma}

Rather than give the routine proof, it will be profitable to focus on a result that gives somewhat more information, rendering \Cref{pr:ifprod} applicable to the ideals $\cL^{p,q}$.

\begin{remark}\label{re:useseqs}
  Throughout the ensuing discussion we will make mostly tacit use of the general principle \cite[Theorem 2.8]{sim_tr} whereby it is enough to prove the analogous statements about ideals of {\it sequence} spaces (or, equivalently, work with operators diagonalized under a single, fixed basis).

  We will thus often abuse the notation (and language) in asserting that {\it sequences} belong to ideals such as $\cL^{p,q}$; this means that the corresponding diagonal {\it operators} do, etc.
\end{remark}

\begin{proposition}\label{pr:pqareord}
  Let
  \begin{equation*}
    (p_0,q_0) \le (p_1,q_1)\text{ lexicographically},
  \end{equation*}
  with
  \begin{equation*}
    1<p_i<\infty,\ 1\le q_i\le \infty
    \quad\text{or}\quad
    (p_i,q_i)\in \{(1,1),\ (\infty,\infty),\ (\infty,1)\}.
  \end{equation*}
  We then have
  \begin{equation}\label{eq:p01q01-approx}
      \cL_0^{p_0,q_0}\preceq_{\cK} \cL_0^{p_1,q_1}
    \end{equation}
    in the notation of \Cref{eq:orderincl}. 

    
\end{proposition}
\begin{proof}
  Recall from \Cref{res:whennot0} \Cref{item:whenneed0} that only when $q=\infty$ do we have a distinction between $\cL$ and $\cL_0$, so we will occasionally drop the `0' subscript. The transitivity of `$\preceq_{\cK}$' makes it sufficient to argue in stages.



  \begin{enumerate}[(I)]
  \item\label{item:qto1} {\bf : $\cL^{p,1}\preceq_{\cK} \cL_0^{p,q}$.} When $p=\infty$ this is simply
    \begin{equation*}
      \cL^{\infty,1}
      \preceq_{\cK}
      \cL_0^{\infty,\infty} = \cL_0^{\infty} = \cK(H),
    \end{equation*}
    which follows, say, from the characterization of approximability given in \Cref{th:multcl}.


    Because furthermore $q=\infty$ is covered by \Cref{item:inftoq} below for $p>1$ and avoided by fiat for $p=1$, we can assume
    \begin{equation*}
      1<p,q<\infty.
    \end{equation*}
    (there being nothing to prove for $q=1$). Suppose, then, that the positive sequence $(\mu_n)$ satisfies
    \begin{equation*}
      \sum_{n\ge 1}\frac{n^{\alpha}\mu_{n-1}}{n}<\infty
    \end{equation*}
    so that it belongs to $\cL^{p,1}$ (in the sense of \Cref{re:useseqs}) by \Cref{pr:hardy}. We can now write
    \begin{equation}\label{eq:havefact}
      \mu_{n-1} = s_n t_n
      ,\quad
      s_n:= n^{\alpha\frac{1-q}q}\mu_{n-1}^{\frac 1q}
      ,\quad
      t_n:=n^{\alpha\frac{q-1}q}\mu_{n-1}^{\frac{q-1}q}.
    \end{equation}
    Note that
    \begin{equation*}
      \sum \frac{\left(n^{\alpha}s_n\right)^q}n
      =
      \sum \frac{n^{\alpha}\mu_{n-1}}n
      <\infty
    \end{equation*}
    so that $(s_n)\in \cL^{p,q}$. On the other hand, for the {\it conjugate exponent} \cite[Definition 3.4]{rud_rc} $q':=\frac{q}{q-1}$ we also have
    \begin{equation*}
      \sum\frac{t_n^{q'}}n
      =
      \sum \frac{n^{\alpha}\mu_{n-1}}n
      <\infty
    \end{equation*}
    so that $(t_n)\in (\cL^{p,1}:\cL^{p,q})$ by the {\it H\"older inequality} \cite[Theorem 3.5]{rud_rc} for the measure space $\bZ_{>0}$ with $n\in \bZ_{>0}$ given weight $\frac 1n$. The factorization \Cref{eq:havefact} thus settles the present case. 
    
  \item\label{item:p0p1} {\bf : $\cL_0^{p_0,\infty} \preceq_{\cK} \cL^{p_1,1}\text{ for }1 < p_0 < p_1$.}

    The cases in question are representable uniformly (taking \Cref{re:useseqs} into account and shifting the index of the positive sequence $(\mu_n)$ for convenience):
    \begin{equation*}
      (\mu_n)\in \cL_0^{p_0,\infty} \xLeftrightarrow{\quad} n^{\alpha_0}\mu_n\xrightarrow[n]{\quad}0,\ \text{i.e. $\mu_n=o(n^{-\alpha_0})$}
    \end{equation*}
    and
    \begin{equation*}
      (s_n)\in \cL^{p_1,1} \xLeftrightarrow{\quad} \sum_{n\ge 1}\frac {n^{\alpha_1}s_{n}}n<\infty,
    \end{equation*}
    with $\alpha_i:=\frac 1{p_i}$.

    Factor
    \begin{equation}\label{eq:ust}
      \cL_0^{p_0,\infty} \ni \mu_n = s_n\cdot t_n
      \quad\text{with}\quad
      n^{\alpha_0-\alpha_1}t_n = \frac{n^{\alpha_0}\mu_n}{n^{\alpha_1}s_n} = o(1)
    \end{equation}
    while at the same time
    \begin{equation*}
      \sum\frac{n^{\alpha_1}s_n}n<\infty\xRightarrow{\quad} (s_n)\in \cL^{p_1,1}. 
    \end{equation*}
    We then have
    \begin{align*}
      (y_n)\in \cL^{p_1,1} &\xRightarrow[\quad]{} y_n=o(n^{-\alpha_1})\\
                           &\xRightarrow[\quad]{\text{\Cref{eq:ust}}} n^{\alpha_0} y_n t_n=o(1)\\
                           &\xRightarrow[\quad]{} y_n t_n = o(n^{-\alpha_0}),\\
    \end{align*}
    meaning that $(t_n)\in (\cL_0^{p_0,\infty}:\cL^{p_1,1})$.

    
  \item\label{item:inftoq} {\bf : $\cL^{p,q} \preceq_{\cK} \cL_0^{p,\infty}$ for $1\le q<\infty$.} The cases $p=1,\infty$ are prohibited and covered by \Cref{item:qto1} above respectively (since then $q=1,\infty$ as well, etc.), so we can assume that $p$ is in the bulk range $(1,\infty)$. 

    The rest is virtually tautological. Consider a non-increasing positive sequence $(\mu_n)_n$ satisfying the defining property of $\cL^{p,q}$ (cf. \Cref{eq:pqnorm'}), shifted for convenience so as to disregard the $0^{th}$ term:
    \begin{equation*}
      \sum_{n\ge 1}\frac{\left(n^{\alpha}\mu_{n}\right)^q}{n} <\infty
    \end{equation*}
    (positivity is harmless: if $\mu_n$ is eventually 0 then the claim is even easier to prove).

    The bounded sequence $(t_n):=\left(n^{\alpha} \mu_n\right)$ belongs to $(\cL^{p,q}:\cL_0^{p,\infty})$ (\Cref{re:useseqs} being operative), because multiplication by it converts any $O(n^{-\alpha})$ sequence (in {\it big-oh notation} \cite[\S 3.2]{clrs_alg-4e}) into $(\mu_n)\in \cL^{p,q}$. The ratio $\frac{\mu_n}{t_n}$ satisfies
    \begin{equation*}
      n^{\alpha}\frac{\mu_n}{t_n}=o(1)
      \Longrightarrow
      \frac{\mu_n}{t_n}\in \cL_0^{p,\infty},
    \end{equation*}
    so that
    \begin{equation*}
      (\mu_n) = \left(\frac{\mu_n}{t_n}\right)\cdot (t_n)
      \in \cL_0^{p,\infty}\cdot (\cL^{p,q}:\cL_0^{\infty}).
    \end{equation*}

  \item {\bf : $\cL^{1} \preceq_{\cK} \cL^{p,1}$ for $1<p<\infty$.} Had the statement not forbidden the use of $\cL_0^{1,\infty}$, the transition could have been effected in two stages, via \Cref{item:p0p1} and \Cref{item:inftoq}. As things stand, the sequence-based argument still goes through if in place of
    \begin{equation*}
      \cL_0^{1,\infty}
      \xlrsquigarrow{\qquad}
      \sigma_n=o(\log n)
    \end{equation*}
    we use the sequence space defined by $\mu(n)=o(n^{-1})$.
  \end{enumerate}
\end{proof}

Returning to the issue of comparing topologies, the phenomenon of coincidence on bounded sets will again be familiar from one's acquaintance with the ($\sigma$-)strong topologies (e.g. \cite[Lemma II.2.5]{tak1}, \cite[Theorem II.7]{ake_dual}, etc.). Recall from \Cref{def:sn} that sn-ideals are, by default, assumed to be contained in the ideal $\cK$ of compact operators. 

\begin{theorem}\label{th:topssamebdd}
  Let $H$ be a Hilbert space, $\cB:=\cB(H)$, $\cK:=\cK(H)$ the ideal of compact operators, and $\cI\trianglelefteq \cK$ a generic sn-ideal.

  \begin{enumerate}[(1)]

  \item\label{item:topssamebdd} $\mu_{\cI}$ and $\mu^*_{\cI}$ all have the same bounded sets (namely the norm-bounded ones).

  \item\label{item:topssamebdd-all} For approximable $\cI$ the multiplier topologies $\mu_{\cI}$ all coincide on bounded subsets of $\cB$ with the $\sigma$-strong topology $\mu_2$ of \Cref{not:mutop}.

  \item\label{item:topssamebdd-all*} Similarly, for approximable $\cI$ all $\mu_{\cI}^*$ coincide on bounded sets with the $\sigma$-strong$^*$-topology $\mu_2^*$.

  \item\label{item:topssamebdd-fn} If $\cI$ is approximable and
    \begin{equation}\label{eq:l1getback}
      \cL^1 = \cI (\cL^1:\cI)
    \end{equation}
    then $\mu_{\cI}$ and $\mu_{\cI}^*$ have the same continuous dual $\cL^1\cong \cB_*$ (i.e. the weak$^*$-continuous functionals on $\cB$).

  \end{enumerate}
\end{theorem}
\begin{proof}
  
  \begin{enumerate}[]

  \item {\bf \Cref{item:topssamebdd}} Condition \Cref{eq:rk1} shows that $\mu_{\cI}$- or $\mu_{\cI}^*$-boundedness entails pointwise boundedness on $H$, hence norm-boundedness \cite[\S IV.2, Corollary 2]{rr_tvs}. The converse, on the other hand, follows from \Cref{eq:axbaxb}.

  \item {\bf \Cref{item:topssamebdd-all}} The $\sigma$-strong topology $\mu_2$ is one of the $\mu_I$, hence the last claim (assuming the rest).

    To prove the main assertion, note first that the ideal $\cF=\cF(H)\subset \cB(H)$ of finite-rank operators is contained in all $\cI$ under consideration and by assumption dense for each of the corresponding symmetric norms $\|\cdot\|_{\cI}$. It follows that on bounded subsets of $\cB(H)$ the left-multiplier topology induced by $\|\cdot\|_{\cI}$ on $\cF$ coincides with that induced by multiplication on all of $\cI$ (as in \cite[Lemma 2.3.6]{wo}, say).
    
    Observe next that the topologies induced on every
    \begin{equation*}
      \cF_n:=\{\text{operators of rank }\le n\}
    \end{equation*}
    by the norms $\|\cdot\|_{\cI}$ all coincide for every individual $n$, and each $\cF_n$ is closed under left and right multiplication by arbitrary $T\in \cB(H)$.
    
  \item {\bf \Cref{item:topssamebdd-all*}} follows immediately from \Cref{item:topssamebdd-all}, since each $\mu^*$ is the weakest locally convex topology stronger than both the corresponding $\mu$ and the latter's image under the $*$ operator.

  \item {\bf \Cref{item:topssamebdd-fn}} We have the strength-comparability relations
    \begin{equation*}
      \mu_1\preceq \mu_{\cI}\preceq \mu_{\cI}^*\preceq \mu_{\cK}^*=\text{Mackey}:
    \end{equation*}
    the first by \Cref{eq:l1getback} and \Cref{pr:ifprod}, the second generically (every $\mu$ being at most as strong as its corresponding $\mu^*$), the third by the approximability assumption and \Cref{cor:allweakerthanmuinf}, and the last equality by \cite[Corollary 2.8]{tay_strict}.

    Because $\mu_1$ is the left-multiplier topology induced by multiplication on $\cL^1(H)\cong \cB(H)_*$, it is stronger than the weak$^*$ topology on $\cB$, induced \cite[Theorem I.8.6.1]{blk_oa} by the seminorms
    \begin{equation*}
      \cB\ni T\xmapsto{\quad\omega_X\quad}\left|\mathrm{tr}~TX\right|,\ X\in \cL^1.
    \end{equation*}
    Since the weak$^*$ and Mackey topologies are, respectively, the weakest and strongest \cite[\S 21.4 (3)]{koeth_tvs-1} locally convex topologies with continuous dual $\cB_*\cong \cL^1$, the conclusion follows.
  \end{enumerate}
\end{proof}

The approximability assumption is crucial to items \Cref{item:topssamebdd-all} and \Cref{item:topssamebdd-all*} of \Cref{th:topssamebdd}:

\begin{example}\label{ex:needapprox}
  Consider an orthonormal basis $(e_n)_{n\in \bZ_{\ge 0}}$ for a Hilbert space $H$, and the operators
  \begin{equation*}
    E_n:=\text{projection on }\mathrm{span}\{e_i,\ 0\le i\le n\}.
  \end{equation*}
  The (bounded!) sequence $(E_n)$ converges to the identity strongly$^*$ and hence in the topology $\mu_2^*$, but not, say, in the topology $\mu_{1,\infty}$ induced by \Cref{eq:1infnorm}:
  \begin{equation*}
    (1-E_n)
    \cdot
    \mathrm{diag}\left(1,\ \frac 12,\ \cdots,\ \frac 1n,\ \cdots\right)
    =
    \mathrm{diag}\left(0,\ \cdots,\ 0,\ \frac 1{n+1},\ \frac 1{n+2},\ \cdots\right)
  \end{equation*}
  does not $\|\cdot\|_{1,\infty}$-converge to 0.
\end{example}

\begin{remark}\label{re:needapprox}
  \Cref{ex:needapprox} also shows that one cannot, in general, expect comparability relations $\mu_{\cI}\preceq \mu_{\cJ}$ (in the style of \Cref{pr:ifprod}) in general, for sn-ideal inclusions $\cI\subseteq \cJ$: the example addresses this for $\cI=\cL^{1,\infty}$ and $\cJ=\cK$.
\end{remark}

As it happens, \Cref{th:topssamebdd} \Cref{item:topssamebdd-fn} also requires approximability crucially:

\begin{example}\label{ex:dixtr}
  Consider any one of the {\it Dixmier traces} $\mathrm{Tr}_{\omega}$ on $\cL^{1,\infty}=\cL^{1,\infty}(H)$ in \cite[\S $4.2.\beta$, Definition 2]{Con94}: these are easily seen to be $\|\cdot\|_{1,\infty}$-continuous, but are not weak$^*$-continuous and vanish on the approximable portion $\cL_0^{1,\infty}\subset \cL^{1,\infty}$.
  
  One can then produce a non-weak$^*$-, $\mu^*_{1,\infty}$-continuous functional on $\cB=\cB(H)$ by
  \begin{equation*}
    \cB\ni X\xmapsto[]{\quad \varphi_{\omega,T}\quad} \mathrm{Tr}_{\omega} TX,\quad T\in \cL^{1,\infty}:    
  \end{equation*}
  the $\mu^*_{1,\infty}$-continuity follows from that of multiplication by $T$ and the $\|\cdot\|_{1,\infty}$-continuity of $\mathrm{Tr}_{\omega}$, while the lack of weak$^*$-continuity can be ensured by taking $T$ outside the kernel of $\mathrm{Tr}_{\omega}$, so that
  \begin{itemize}
  \item $\varphi_{\omega,T}$ vanishes on the weak$^*$-dense ideal $\cL_0^{1,\infty}$ of $\cB$;
  \item but not on all of $\cB$ itself, because it does not vanish on $1$. 
  \end{itemize}
\end{example}

\Cref{pr:ifprod} leverages an ideal inclusion into an ordering for the corresponding multiplier topologies. The following partial converse generalizes (and gives some surrounding context for) \Cref{ex:yeadex}: it specializes back to the latter upon selecting $\cI:=\cK$ and $\cJ:=\cL^2$.

\begin{theorem}\label{th:yead}
  For two approximable sn-ideals $\cI,\cJ\trianglelefteq \cK(H)$ we have
  \begin{equation*}
    \cI\not\subseteq \cJ\xRightarrow{\quad} \mu_{\cI}\not\preceq \mu^*_{\cJ}. 
  \end{equation*}
\end{theorem}
\begin{proof}
  We may as well assume the underlying Hilbert space $H$ is infinite-dimensional, or there is nothing to prove.
  
  Let $\Phi_{\cI}$ and $\Phi_{\cJ}$ be the sn-functions corresponding to the two ideals. The assumed non-inclusion entails the existence of a non-increasing positive $c_0$-sequence $(\mu_n)_{n\in \bZ_{\ge 0}}$ belonging to the domain of $\Phi_{\cI}$ but not that of $\Phi_{\cJ}$ (\Cref{re:salinas}).

  As in \Cref{ex:yeadex}, let $E_n$, $n\in \bZ_{\ge 0}$ be (non-zero) mutually-orthogonal projections on $H$ and $T\in \cJ$. We cannot have
  \begin{equation}\label{eq:alllarge}
    \left\vvvert \frac 1{\mu_n} TE_n\right\vvvert_{\cJ}\ge C>0,\ \forall \text{ large } n,
  \end{equation}
  because in that case the operator
  \begin{equation*}
    T_{\ge n_0}:=\sum_{n\ge n_0}TE_n\text{ (strong limiting sum) for some }n_0
  \end{equation*}
  would
  \begin{itemize}
  \item belong to $\cJ$, being a product between $T$ and a projection $\sum_{n\ge n_0}E_n$;
  \item while at the same time having (some) characteristic numbers that dominate $C\mu_{n}$, $n\ge n_0$.
  \end{itemize}
  The latter condition means by assumption that some sequence of characteristic numbers of $T_{\ge n_0}$ lies outside the domain of $\Phi_{\cJ}$, and hence \cite[\S III.4, Theorem 4.3]{gk_lin} $T_{\ge n_0}\not\in \cJ$. The contradiction shows that indeed \Cref{eq:alllarge} does not obtain, and hence $\{E_n\}_n$ has $0$ as a $\mu_{\cJ}^*$-cluster point (there is no distinction here between $\mu$ and $\mu^*$, the $E_n$ being self-adjoint).
  
  On the other hand, again mimicking \Cref{ex:yeadex}, a diagonal operator $D$ scaling a unit vector $\xi_n\in\mathrm{range}(E_n)$ by $\mu_n$ belongs to $\cI$, and
  \begin{equation*}
    \frac 1{\mu_n}E_n D\xi_n = \xi_n
  \end{equation*}
  shows that $\left\{\frac 1{\mu_n}E_n\right\}_n$ does {\it not} have $0$ as a $\mu_{\cI}$-cluster point. 
\end{proof}

Collecting a number of the preceding general results, we now have the following lattice of multiplier topologies (recall \Cref{not:mutop}).

\begin{theorem}\label{th:lpqmultops}
  Let 
  \begin{equation}\label{eq:pqrng}
    (p,q)
    \in
    (1,\infty)\times [1,\infty]
    \ \sqcup\ 
    \{(1,1),\ (\infty,\infty),\ (\infty,1)\}.
  \end{equation}
  \begin{enumerate}[(1)]
    
  \item\label{item:strord} The topologies $\mu_{p,q\mid 0}$ are ordered strictly by strength, increasingly and lexicographically with \Cref{eq:pqrng}. The same goes for the topologies $\tensor[_r]{\mu}{_{p,q\mid 0}}$ and $\mu_{p,q\mid 0}^*$.

    In fact, 
    \begin{equation*}
      (p_0,q_0) \lneq (p_1,q_1)
      \xRightarrow{\quad}
      \mu_{p_1,q_1},\ \tensor[_r]{\mu}{_{p_1,q_1}}
      \npreceq
      \mu^*_{p_0,q_0}.
    \end{equation*}

  \item\label{item:pqsameonbdd} All $\mu_{p,q\mid 0}$ coincide on bounded subsets of $\cB(H)$, as do all $\tensor[_r]{\mu}{_{p,q\mid 0}}$ and all $\mu_{p,q\mid 0}^*$.

  \item\label{item:pqsamefunc} $\mu_{p,q\mid 0}$, $\tensor[_r]{\mu}{_{p,q\mid 0}}$ and $\mu_{p,q\mid 0}^*$ all have the same (norm-)bounded sets and the same (weak$^*$-)continuous functionals.
    
  \end{enumerate}
\end{theorem}
\begin{proof}
  Parts \Cref{item:pqsameonbdd} and \Cref{item:pqsamefunc} follow from \Cref{th:topssamebdd} and \Cref{pr:pqareord}, which ensures that the former's hypotheses are satisfied.

  As for part \Cref{item:strord}, the (possibly non-strict) ordering follows from \Cref{pr:ifprod} and once more \Cref{pr:pqareord}, whereas the negative assertion is a consequence of \Cref{th:yead}.
\end{proof}

\addcontentsline{toc}{section}{References}

\begin{thebibliography}{10}

\bibitem{ake_dual}
Charles~A. Akemann.
\newblock The dual space of an operator algebra.
\newblock {\em Trans. Amer. Math. Soc.}, 126:286--302, 1967.

\bibitem{am_comm}
Michael~F. Atiyah and I.~G. Macdonald.
\newblock {\em Introduction to commutative algebra}.
\newblock Boulder: Westview Press, student economy edition edition, 2016.

\bibitem{bl_interp}
J\"{o}ran Bergh and J\"{o}rgen L\"{o}fstr\"{o}m.
\newblock {\em Interpolation spaces. {A}n introduction}.
\newblock Grundlehren der Mathematischen Wissenschaften, No. 223.
  Springer-Verlag, Berlin-New York, 1976.

\bibitem{blk_oa}
B.~Blackadar.
\newblock {\em Operator algebras}, volume 122 of {\em Encyclopaedia of
  Mathematical Sciences}.
\newblock Springer-Verlag, Berlin, 2006.
\newblock Theory of $C^*$-algebras and von Neumann algebras, Operator Algebras
  and Non-commutative Geometry, III.

\bibitem{Con94}
Alain Connes.
\newblock {\em Noncommutative geometry}.
\newblock Academic Press, Inc., San Diego, CA, 1994.

\bibitem{clrs_alg-4e}
Thomas~H. Cormen, Charles~E. Leiserson, Ronald~L. Rivest, and Clifford Stein.
\newblock {\em Introduction to algorithms}.
\newblock Cambridge, MA: MIT Press, 4th edition edition, 2022.

\bibitem{dp_latt}
B.~A. Davey and H.~A. Priestley.
\newblock {\em Introduction to lattices and order.}
\newblock Cambridge: Cambridge University Press, 2nd ed. edition, 2002.

\bibitem{dixw}
Jacques Dixmier.
\newblock {\em von {N}eumann algebras}, volume~27 of {\em North-Holland
  Mathematical Library}.
\newblock North-Holland Publishing Co., Amsterdam-New York, 1981.
\newblock With a preface by E. C. Lance, Translated from the second French
  edition by F. Jellett.

\bibitem{dfww_comm}
Ken Dykema, Tadeusz Figiel, Gary Weiss, and Mariusz Wodzicki.
\newblock Commutator structure of operator ideals.
\newblock {\em Adv. Math.}, 185(1):1--79, 2004.

\bibitem{finch_const-1}
Steven~R. Finch.
\newblock {\em Mathematical constants}, volume~94 of {\em Encycl. Math. Appl.}
\newblock Cambridge: Cambridge University Press, 2003.

\bibitem{gk_lin}
I.~C. Gohberg and M.~G. Kre\u{\i}n.
\newblock {\em Introduction to the theory of linear nonselfadjoint operators}.
\newblock Translations of Mathematical Monographs, Vol. 18. American
  Mathematical Society, Providence, R.I., 1969.
\newblock Translated from the Russian by A. Feinstein.

\bibitem{koeth_tvs-1}
Gottfried K\"{o}the.
\newblock {\em Topological vector spaces. {I}}.
\newblock Die Grundlehren der mathematischen Wissenschaften, Band 159.
  Springer-Verlag New York, Inc., New York, 1969.
\newblock Translated from the German by D. J. H. Garling.

\bibitem{lfv_hardy}
Pascal Lef\`evre.
\newblock Weighted discrete hardy’s inequalities, 2020.
\newblock available at \url{https://hal.science/hal-02528265} (accessed
  2023-05-30).

\bibitem{lfv_comm}
Pascal Lef\`evre.
\newblock personal communication, 2023-05-30.

\bibitem{lt_cls}
Joram Lindenstrauss and Lior Tzafriri.
\newblock {\em Classical {Banach} spaces. 1: {Sequence} spaces. 2. {Function}
  spaces.}
\newblock Class. Math. Berlin: Springer-Verlag, repr. of the 1977 and 1979 ed.
  edition, 1996.

\bibitem{ok_hardy}
B~Opic and Alois Kufner.
\newblock {\em Hardy-type inequalities}.
\newblock Harlow: Longman Scientific \&| Technical; New York: John Wiley \&|
  Sons, Inc., 1990.

\bibitem{ped_aut}
Gert~K. Pedersen.
\newblock {\em {$C^*$}-algebras and their automorphism groups}.
\newblock Pure and Applied Mathematics (Amsterdam). Academic Press, London,
  2018.
\newblock Second edition of [ MR0548006], Edited and with a preface by S\o ren
  Eilers and Dorte Olesen.

\bibitem{rr_tvs}
A.~P. Robertson and Wendy Robertson.
\newblock {\em Topological vector spaces}, volume~53 of {\em Cambridge Tracts
  in Mathematics}.
\newblock Cambridge University Press, Cambridge-New York, second edition, 1980.

\bibitem{rud_rc}
Walter Rudin.
\newblock {\em Real and complex analysis.}
\newblock New York, NY: McGraw-Hill, 3rd ed. edition, 1987.

\bibitem{sal_symm}
Norberto Salinas.
\newblock Symmetric norm ideals and relative conjugate ideals.
\newblock {\em Trans. Am. Math. Soc.}, 188:213--240, 1974.

\bibitem{sim_tr}
Barry Simon.
\newblock {\em Trace ideals and their applications}, volume 120 of {\em Math.
  Surv. Monogr.}
\newblock Providence, RI: American Mathematical Society (AMS), 2nd ed. edition,
  2005.

\bibitem{tak1}
M.~Takesaki.
\newblock {\em Theory of operator algebras. {I}}, volume 124 of {\em
  Encyclopaedia of Mathematical Sciences}.
\newblock Springer-Verlag, Berlin, 2002.
\newblock Reprint of the first (1979) edition, Operator Algebras and
  Non-commutative Geometry, 5.

\bibitem{tay_strict}
Donald~Curtis Taylor.
\newblock The strict topology for double centralizer algebras.
\newblock {\em Trans. Amer. Math. Soc.}, 150:633--643, 1970.

\bibitem{trev_tvs}
Fran\c{c}ois Tr\`eves.
\newblock {\em Topological vector spaces, distributions and kernels}.
\newblock Dover Publications, Inc., Mineola, NY, 2006.
\newblock Unabridged republication of the 1967 original.

\bibitem{wo}
N.~E. Wegge-Olsen.
\newblock {\em {$K$}-theory and {$C^*$}-algebras}.
\newblock Oxford Science Publications. The Clarendon Press, Oxford University
  Press, New York, 1993.
\newblock A friendly approach.

\bibitem{yead_mack}
F.~J. Yeadon.
\newblock A note on the {M}ackey topology of a von {N}eumann algebra.
\newblock {\em J. Math. Anal. Appl.}, 45:721--722, 1974.

\end{thebibliography}

\def\polhk#1{\setbox0=\hbox{#1}{\ooalign{\hidewidth
  \lower1.5ex\hbox{`}\hidewidth\crcr\unhbox0}}}
  \def\polhk#1{\setbox0=\hbox{#1}{\ooalign{\hidewidth
  \lower1.5ex\hbox{`}\hidewidth\crcr\unhbox0}}}
  \def\polhk#1{\setbox0=\hbox{#1}{\ooalign{\hidewidth
  \lower1.5ex\hbox{`}\hidewidth\crcr\unhbox0}}}
  \def\polhk#1{\setbox0=\hbox{#1}{\ooalign{\hidewidth
  \lower1.5ex\hbox{`}\hidewidth\crcr\unhbox0}}}
  \def\polhk#1{\setbox0=\hbox{#1}{\ooalign{\hidewidth
  \lower1.5ex\hbox{`}\hidewidth\crcr\unhbox0}}}

\Addresses

\end{document}